\documentclass[12pt]{amsart}

\usepackage{amsmath, amscd, graphicx, latexsym, hyperref, rlepsf, times, enumerate}

\oddsidemargin .25in \theoremstyle{plain}

\newtheorem{Thm}{Theorem}

\newtheorem{Lem}[Thm]{Lemma}

\newtheorem{Prop}[Thm]{Proposition}

\newtheorem{Rem}[Thm]{Remark}

\newcommand{\genus}{\operatorname{genus}}

\newcommand{\rot}{\operatorname{rot}}

\newcommand{\tb}{\mathop{\rm tb}\nolimits}

\newcommand{\OB}{\mathcal{OB}}

\newcommand{\tr}{\operatorname{tr}}
\newcommand{\SL}{\operatorname{SL}}

\def\v{\vskip.12in}

\def\d{\delta}
\def\g{\gamma}

\begin{document}

\title{Canonical contact structures on some singularity links}

\author{Mohan Bhupal and Burak Ozbagci}


\address{Department of Mathematics,  METU, Ankara, Turkey, \newline bhupal@metu.edu.tr}

\address{Department of Mathematics, Ko\c{c} University, Istanbul, Turkey
\newline bozbagci@ku.edu.tr}


\begin{abstract}

We identify the canonical contact structure on the link of a simple
elliptic or cusp singularity by drawing a Legendrian handlebody
diagram of one of its Stein fillings. We also show that the
canonical contact structure on the link of a numerically Gorenstein
surface singularity is trivial considered as a real plane
bundle.

\end{abstract}

\maketitle

\section{Introduction}

There is a canonical (a.k.a. Milnor fillable)  contact structure
$\xi_{can}$ on the link of an isolated complex surface singularity,
which is unique up to isomorphism \cite{cnp}. It is known that a
Milnor fillable contact structure is Stein fillable \cite{bd} and
universally tight \cite{lo} (note that a Stein fillable contact
structure is not necessarily universally tight \cite[page 670]{go}).
Moreover there are only finitely many isomorphism classes of Stein
fillable contact structures on a closed orientable $3$-manifold
\cite[Lemma 3.4]{ay}. In favorable circumstances, these facts are
sufficient to pin down the canonical contact structure on a given
singularity link.

In this article we determine the canonical contact structure on the
link of a simple elliptic or cusp singularity by drawing a
Legendrian handlebody diagram (cf. \cite{go,gs}) of one of its Stein
fillings. In addition, we convert this diagram into a surgery
diagram of the canonical contact structure by replacing any
$1$-handle in the Stein filling by a contact $(+1)$-surgery along a
Legendrian unknot in the standard contact $S^3$ (cf.
\cite{dg1,dg2}).

As a byproduct of our constructions, we answer positively, for the
classes of singularities above, the question posed by Caubel, Nemethi
and Popescu-Pampu of whether all Milnor fillable contact structures
on a Milnor fillable $3$-manifold are not just
isomorphic but \emph{isotopic} (cf. \cite[Remark 4.10]{cnp}).

The classes of simple elliptic and cusp singularities are known to
be numerically Gorenstein. Here we show that the canonical contact
structure on the link of such a surface singularity is trivial
considered as a real plane bundle.

The contact structures which appear in this paper are assumed to be
positive and co-orientable.  The reader is advised to turn to
\cite{nemet} for more on the canonical contact structures.

\section{Numerically Gorenstein
surface singularities}

The class of {\em numerically Gorenstein} surface singularities was
introduced by A. Durfee in \cite{dur} as the class of singularities
for which the holomorphic tangent bundle is smoothly trivial on
a punctured neighborhood of the singular point. As Durfee observed
(see Lemma~\ref{tfae}), a surface singularity is numerically
Gorenstein if and only if its canonical bundle is smoothly
trivial.

By contrast, a normal surface singularity is {\em Gorenstein} if and
only if its canonical bundle is holomorphically trivial. This shows
that a Gorenstein normal surface singularity is numerically
Gorenstein, but the converse is not true in general. However, every
numerically Gorenstein normal surface singularity {\em is}
homeomorphic to a Gorenstein one \cite{ppp}, which shows that, in
order to study their canonical contact structures, it suffices to
study those of normal Gorenstein ones.

\begin{Prop} \label {num} The canonical contact structure of a numerically
Gorenstein surface singularity is trivial, considered as a
real plane bundle.
\end{Prop}

Proposition~\ref{num} follows immediately from the following result
that was communicated to us by the referee, whose proof is due to P.
Popescu-Pampu.

\begin{Lem} \label{anti}  Let $Y$ be the link of a normal surface
singularity $(S, p)$. Then the canonical contact structure on $Y$,
viewed as an abstract bundle, is smoothly isomorphic to the
restriction to $Y$ of the anticanonical bundle $K^*_{S}$ of $S$.
\end{Lem}

\begin{proof} Fix an auxiliary Riemannian metric on a
neighborhood of $p \in S$ containing $Y$ and consider the (real) normal bundle
$\nu$ of $Y$ in $S$. Observe that $\nu$ is trivial since $Y$ is oriented.
Hence its complexification $\nu^{\mathbb{C}}$ is trivial as a smooth
complex bundle. Now since $\xi$ is given by complex tangencies to $Y$,
$TS|_Y = \xi \oplus {\nu}^{\mathbb{C}}$. Thus, as smooth complex bundles,
$$\Lambda^2TS|_Y \simeq \xi \otimes_{\mathbb{C}} {\nu}^\mathbb{C}\simeq \xi.$$
The lemma now follows from the fact that
$\Lambda^2 TS$ is the dual of $\Lambda^2 T^*S=K_S$.
\end{proof}

\begin{proof}[Proof of Proposition~\ref{num}]
Suppose that $(S, p)$ is numerically Gorenstein. Then $K_S$
restricted to $Y$ is smoothly trivial as a complex line bundle.
Hence $\xi$ is also smoothly trivial as a real plane bundle by
Lemma~\ref{anti}.
\end{proof}

In fact, for this dimension the converse of Proposition~\ref{num} is true, as stated in Lemma
1.1 of Durfee's paper quoted before:

\begin{Lem}[Durfee]  \label{tfae} Let $\zeta$ be a two-dimensional complex
bundle over a CW complex $X$ with $H^i(X) = 0$ for $i > 3$. Then the
following conditions are equivalent:
\begin{enumerate}[(a)]
\item $\zeta$ is trivial.
\item $\zeta$ is stably trivial.
\item The first Chern class $c_1(\zeta)$ is zero.
\item The second exterior power ${\Lambda}^2\zeta$ is a
trivial line bundle.
\end{enumerate}
\end{Lem}

In this article, we consider canonical contact structures on links of
simple elliptic and cusp singularities. As every simple elliptic and cusp singularity
is minimally elliptic, it follows by Laufer (\cite{la}) that every such singularity is
Gorenstein and hence {\it a fortiori} numerically Gorenstein.

\section{Simple elliptic and cusp singularities}

Part of the discussion in this article is based on the topological
characterization of singularity links which fiber over the circle
given by Neumann \cite{neum}: A singularity link fibers over the
circle if and only if it is a torus bundle over the circle whose
monodromy $A\in \SL(2,\mathbb{Z})$ is either parabolic, i.e., $\tr
(A) =2$ or hyperbolic with $\tr (A) \geq 3$. Moreover these links
correspond precisely to simple elliptic and cusp singularities,
respectively.

On the other hand, the classification of tight contact structures on
torus bundles by Honda \cite{h2} coupled with a theorem of Gay
\cite{gay} implies that on any torus bundle over the circle there is
a unique Stein fillable universally tight contact structure up to
isomorphism. We conclude that on a singularity link which fibers
over the circle, the canonical contact structure is the unique
universally tight Stein fillable contact structure.

In the following subsections we determine the canonical contact
structures on simple elliptic and cusp singularities---each of which
requires a separate treatment.

\subsection{Simple elliptic singularities}

For a positive integer $n$, let $(S_n,p)$ denote the complex surface
singularity whose minimal resolution consists of a single elliptic
curve of \emph{negative} self-intersection number $-n$. Such
singularities are known as {\em simple elliptic singularities}. The
link $Y_n$ of such a singularity is a torus bundle over the circle
with parabolic monodromy $A\in \SL(2,\mathbb{Z})$. Moreover, the
$3$-manifold $Y_n$ also admits an oriented circle fibration over the
torus with Euler number $-n$.

An open book decomposition $\OB_n$ of $Y_n$ transverse to the circle
fibration was constructed  in \cite{eo} such that the binding
consists of $n$ distinct \emph{positively oriented} circle fibres,
the page is a torus with $n$ boundary components and the monodromy
is the product of $n$ boundary parallel right-handed Dehn twists.
 According to \cite[Theorem 2.1]{np},
 there is a Milnor open book $\widetilde{\OB}_n$ on $Y_n$
whose binding agrees with the binding of $\OB_n$. On the other hand,
by \cite[Theorem 4.6]{cnp}, any two horizontal open books on $Y_n$
with the same binding are isomorphic. It follows that $\OB_n$ is in
fact a Milnor open book and hence it supports the canonical contact
structure which is certainly Stein fillable and universally tight.

\begin{Prop} There are precisely two distinct isotopy classes of Stein
fillable universally tight contact structures on $Y_n$. These are
given as the boundaries of the two distinct Stein surfaces shown in
Figure~\ref{disk}. Both of these contact structures represent the
isomorphism class of the  canonical contact structure $\xi_{can}$ on
$Y_n$.
\end{Prop}

\begin{figure}[ht]
  \relabelbox \small {
  \centerline{\epsfbox{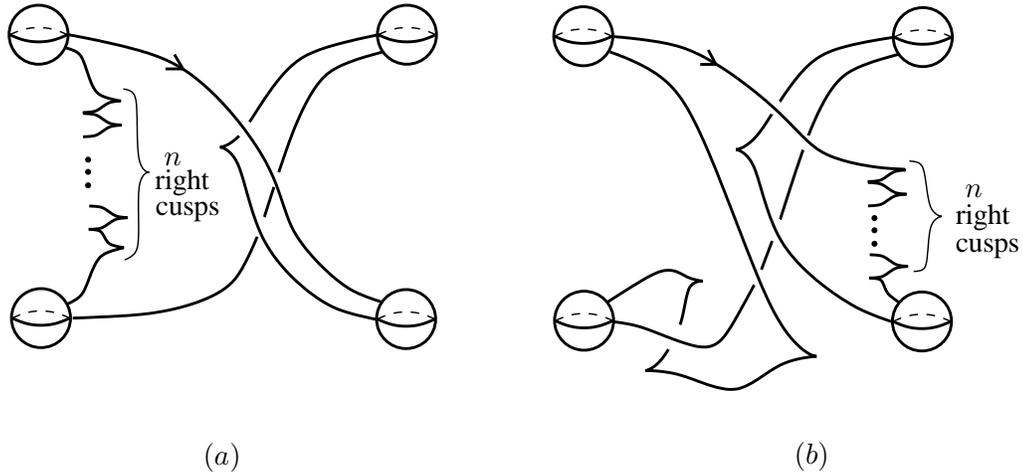}}}

\relabel{a}{$n$}

\relabel{b}{right}

\relabel{c}{cusps}

\relabel{d}{$n$}

\relabel{e}{right}

\relabel{f}{cusps}

\relabel{1}{$(a)$}

\relabel{2}{$(b)$}

  \endrelabelbox
        \caption{Legendrian handlebody diagrams for two distinct Stein structures on
        $X_n$, both inducing the canonical contact structure $\xi_{can}$ on $Y_n$, up to isomorphism. (We
        identify the spheres aligned \emph{horizontally} in both diagrams (a) and (b).)}
        \label{disk}
\end{figure}

\begin{proof}  There are $n+1$ distinct isotopy classes of Stein
fillable contact structures on $Y_n$, two of which are universally
tight according to Honda's classification \cite{h2}. Let $X_n$
denote the oriented $D^2$-bundle over $T^2$ with Euler number $-n$,
a smooth handlebody diagram of which is depicted in \cite[Figure 36
(a)]{go}. In order to obtain Legendrian handlebody diagrams of
distinct Stein surfaces inducing the aforementioned distinct contact
structures on $Y_n$, we use Gompf's trick illustrated in
\cite[Figure 36 (b), (c)]{go}:  We just convert the standard smooth
handlebody diagram of $X_n$ (which consists of two $1$-handles and
one $2$-handle) into a Legendrian handlebody diagram by Legendrian
realizing the attaching circle of the $2$-handle. It is easy to see
that any integer which belongs to the set $$\{-n, -(n-2), \ldots,
n-2, n\}$$ appears as the rotation number of some Legendrian
realization.

According to \cite[Proposition 2.3]{go}, the first Chern class of
the resulting Stein structure on the smooth $4$-manifold $X_n$ is
represented by a cocycle whose value on the homology class induced
by the oriented Legendrian knot (generating $H_2(X_n;
\mathbb{Z})\cong \mathbb{Z}$) is given by the rotation number of
this Legendrian knot. Therefore by \cite[Theorem 1.2]{lm}, the
induced contact structures on $Y_n$ are all pairwise nonisotopic for
all distinct rotation numbers that are listed above.  Moreover, all
of these contact structures except the two that are induced by the
Stein surfaces depicted in Figure~\ref{disk} are virtually
overtwisted by \cite[Proposition 5.1]{go}.

The rotation numbers of the Legendrian knots in Figure~\ref{disk}(a)
and Figure~\ref{disk}(b) are $-n$ and $n$, respectively, with the
indicated orientations. Here if one reverses the orientation of the
Legendrian knot, the sign of the rotation number as well as the sign
of the second homology class induced by this knot in
$H_2(X_n;\mathbb{Z})$ gets reversed. In conclusion, the two extreme
cases where the rotation number of the Legendrian realization of the
attaching circle of the $2$-handle takes its minimal possible value
$-n$ and maximum possible value $n$, respectively, must induce the
two universally tight Stein fillable contact structures on $Y_n$ in
Honda's classification.

To prove the last statement in the proposition, we first observe
that one of two nonisotopic Stein fillable universally tight contact
structures on $Y_n$ is $\xi_{can}$. Let  $\overline{\xi}_{can}$
denote the $2$-plane field  $\xi_{can}$ with the opposite
orientation.  Recall that $\xi_{can}$ is supported by the Milnor
open book $\OB_n$. By reversing the orientation of the page (and
hence the orientation of the binding) of $\OB_n$  we get another
open book $\overline{\OB}_n$ on $Y_n$. The open book
$\overline{\OB}_n$ is in fact isomorphic to $\OB_n$, since they have
identical pages and the same monodromy map measured with the
respective orientations. To see that $\overline{\OB}_n$ is also
horizontal we simply reverse the orientation of the fibre (to agree
with the orientation of the binding) as well as the orientation of
base $T^2$ of the circle bundle $Y_n$, so that we do not change the
orientation of $Y_n$. In addition we observe that the contact
structure supported by $\overline{\OB}_n$ can be obtained from
$\xi_{can}$ by changing the orientations of the contact planes since
$\overline{\OB}_n$ is obtained from $\OB_n$ by changing the
orientations of the pages. We conclude that $\overline{\xi}_{can}$
is a contact structure on $Y_n$ that is isomorphic (but
\emph{nonisotopic}) to $\xi_{can}$.
\end{proof}

Notice that we can trade the $1$-handles in the handlebody diagram
depicted in Figure~\ref{disk} with contact $(+1)$-surgeries as
described in \cite{dg2} and obtain the contact surgery diagram of
$\xi_{can}$ on $Y_n$ as shown in Figure~\ref{disk-wout}.

\begin{figure}[ht]
  \relabelbox \small {
  \centerline{\epsfbox{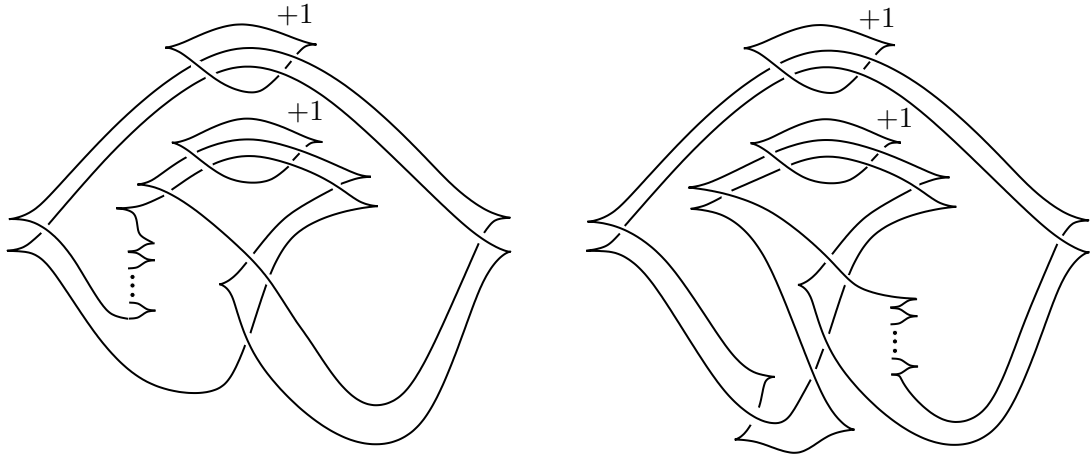}}}

\relabel{b}{$+1$}

\relabel{c}{$+1$}

\relabel{a}{$+1$}

\relabel{d}{$+1$}

  \endrelabelbox
        \caption{Two contact surgery diagrams inducing non-isotopic contact structures that are isomorphic to the canonical contact structure $\xi_{can}$ on $Y_n$.}
        \label{disk-wout}
\end{figure}

\begin{Rem} Simple elliptic singularities are Gorenstein, and by a result
of Seade \cite{se}, the canonical class of any smoothing of $S_n$ is
trivial. Combining this with a result of Pinkham \cite{pi}, which
states that $S_n$ admits a smoothing if and only if $n\leq 9$, we
see that the Euler class of the canonical contact structure on $Y_n$
vanishes if $n\leq 9$. Proposition~\ref{num} implies that the Euler
class of the canonical contact structure on $Y_n$ vanishes in the
case $n > 9$ as well.
\end{Rem}

We would like to point out that the first Chern classes of both
Stein fillings of the contact $3$-manifold $(Y_n, \xi_{can})$
depicted in Figure~\ref{disk} are non-vanishing, although their
restriction to the boundary---the Euler class of
$\xi_{can}$---vanishes.

\begin{Rem} Ohta and Ono \cite{oo} showed that $(Y_n,\xi_{can})$ admits a
strong symplectic filling with vanishing first Chern class if and
only if $n\leq 9$. Moreover they proved that any such filling is
diffeomorphic to a smoothing of the singularity---which is unique
unless $n=8$. For $1 \leq n \leq 9$, a Stein filling of
$(Y_n,\xi_{can})$ with vanishing first Chern class  can be
constructed as a PALF (positive allowable Lefschetz fibration
\cite{ao}) using the $n$-holed torus relation discovered in
\cite{ko}.
\end{Rem}

\subsection{Cusp singularities}

A normal surface singularity having a resolution with exceptional
divisor consisting of a cycle of smooth rational curves, or with
exceptional divisor a single rational curve with a node is called a
{\em cusp singularity}. The link of a cusp singularity also fibers
over the circle with torus fibers and hyperbolic monodromy $A\in
\SL(2,\mathbb{Z})$. Based on a factorization of the monodromy
$A=A(n_1,\ldots, n_k)$, a description of this torus bundle is given
as a circular plumbing graph in Neumann's paper \cite{neum}. For $k
> 1$, the Euler number of the $i$th vertex in the plumbing is equal
to $-n_i$, as depicted in Figure~\ref{circ}$(a)$, where $n_i \geq 2$
for all $i$, and $n_i \geq 3$ for some $i$. For $k=1$, the plumbing
graph consists of a single vertex decorated with an integer
$-n_1\leq -3$ and a loop at this vertex as shown in
Figure~\ref{circ}$(b)$---which corresponds to a self-plumbing of an
oriented circle bundle over $S^2$ with Euler number $-n_1$. To
simplify the notation, we denote the total space of this torus
bundle by $Y_{\overline{n}}$, where $\overline{n}= (n_1,\ldots,
n_k)$.

\begin{figure}[ht]
  \relabelbox \small {
  \centerline{\epsfbox{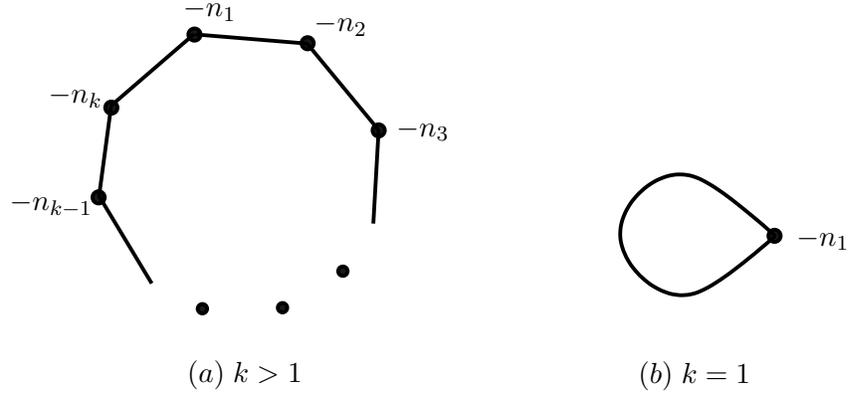}}}

\relabel{g1}{$-n_{k-1}$}

\relabel{g2}{$-n_{k}$}

\relabel{g3}{$-n_1$}

\relabel{g4}{$-n_2$}

\relabel{g5}{$-n_3$}

\relabel{g6}{$-n_1$}

\relabel{A}{$(a)\;  k > 1$ }

\relabel{B}{$(b) \; k = 1$ }

  \endrelabelbox
        \caption{The circular plumbing graph $\Gamma_{\overline{n}}$ for the singularity link $Y_{\overline{n}}$. The vertex
        labeled by $-n_i$ represents a circle bundle over $S^2$ with Euler number $-n_i$. }
        \label{circ}
\end{figure}

\begin{Lem} \label{el} There is a horizontal open book decomposition
$\OB_{\overline{n}}$ on $Y_{\overline{n}}$ with page-genus one,
whose monodromy is given explicitly as a product of some
right-handed Dehn twists.
\end{Lem}

\begin{proof}
We will construct a horizontal open book  $\OB_{\overline{n}}$ with
page-genus one using the methods in \cite{eo}. We first consider the
case $k>1.$ Notice that the $i$th vertex in the plumbing is a circle
bundle over the sphere with Euler number $-n_i \leq -2$. For such a
circle bundle there is a horizontal open book $\OB_i$ where the page
is a sphere with $n_i$ boundary components and the monodromy is the
product of $n_i$ boundary parallel right-handed Dehn twists.

\begin{figure}[ht]
  \relabelbox \small {
  \centerline{\epsfbox{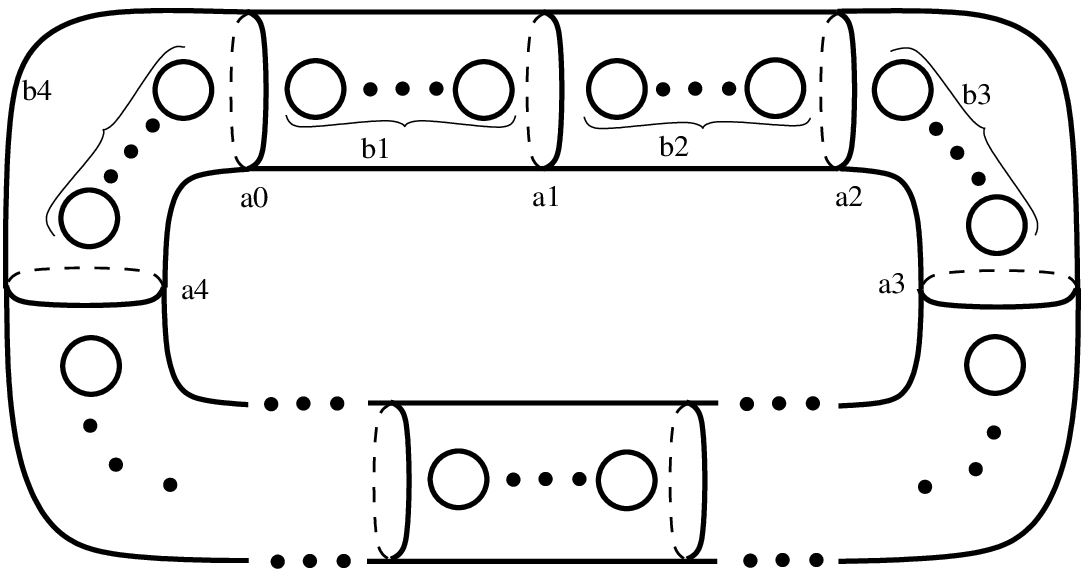}}}

\relabel{b1}{$n_1-2$}

\relabel{b2}{$n_2-2$}

\relabel{b3}{$n_3-2$}

\relabel{b4}{$n_k-2$}

\relabel{a0}{$\d_0$}

\relabel{a1}{$\d_1$}

\relabel{a2}{$\d_2$}

\relabel{a3}{$\d_3$}

\relabel{a4}{$\d_{k-1}$}

  \endrelabelbox
        \caption{The page of the open book $\OB_{\overline{n}}$ of $Y_{\overline{n}}$}
        \label{ellip}
\end{figure}

When two consecutive circle bundles with Euler numbers $-n_i$ and
$-n_{i+1}$ are connected by an edge in the plumbing graph
$\Gamma_{\overline{n}}$, we can ``glue" the corresponding open books
together as follows: First of all, for $i=1,\ldots,k-1$, we glue a
page of $\OB_i$ with a page of $\OB_{i+1}$ using precisely one
boundary component from each page. The Dehn twists along the
identified boundary components merge into a single right-handed Dehn
twist along the resulting curve $\d_i$ (see Figure~\ref{ellip})
after gluing the pages. By gluing all the open books corresponding
to the vertices $i=1,\ldots,k-1$, we get a planar open book. But
since the plumbing graph is circular we need to glue the $\OB_k$
with $\OB_1$ along $\d_0$ so that the resulting page of the open
book $\OB_{\overline{n}}$ is a torus with $\sum_{i=1}^{k} (n_i -2)$
many boundary components as illustrated in Figure~\ref{ellip}. Let
$\g_{i,1},\g_{i,2},\ldots,\g_{i,n_i-2}$ be the boundary-parallel
curves between $\d_{i-1}$ and $\d_{i}$ for $1 \leq i \leq k-1$, and
let $\g_{k,1}, \g_{k,2}, \ldots, \g_{k,n_{k}-2}$ be the boundary-parallel curves between $\d_{k-1}$ and $\d_0$. Then the monodromy of
$\OB_{\overline{n}}$ is given by
$$\prod_{i=0}^{k-1} D(\d_i )\prod_{i=1}^{k}\prod_{j=1}^{n_i-2}
D(\g_{i,j}),$$ where $D(.)$ denotes a right-handed Dehn twist.

We now consider the case $k=1$. Start off with a circle bundle over
a sphere with Euler number $-n_1\leq -3$. As before, for such a
circle bundle there is a horizontal open book $\OB_1$ where the page
is a sphere with $n_1$ boundary components and the monodromy is the
product of $n_1$ boundary parallel right-handed Dehn twists. Adding
a self-plumbing corresponds to gluing together two boundary
components of $\OB_1$ and replacing the corresponding two
boundary-parallel Dehn twists by a single right-handed Dehn twist
along the resulting curve $\d_0$. Let $\g_1,\g_2,\ldots,\g_{n_1-2}$
denote curves parallel to the remaining boundary curves. The
monodromy of the resulting horizontal open book $\OB_{(n_1)}$ is
then given by
$$ D(\d_0)\prod_{i=1}^{n_1-2}D(\g_i),$$

\noindent as shown in Figure~\ref{k1case}.

\begin{figure}[ht]
  \relabelbox \small {
  \centerline{\epsfbox{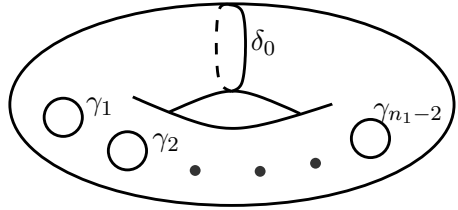}}}

\relabel{1}{$\g_1$}

\relabel{2}{$\g_2$}

\relabel{3}{$\g_{n_1-2}$}

\relabel{4}{$\d_0$}

  \endrelabelbox
        \caption{The page of the open book $\OB_{(n_1)}$ of $Y_{(n_1)}$ }
        \label{k1case}
\end{figure}
\end{proof}

\begin{Rem} The case $k=1$ corresponds to a normal surface singularity having a
resolution with exceptional divisor consisting of a single rational curve
with a node and having self-intersection $-n_1+2\leq -1$. Blowing up
gives the minimal {\em good} resolution with exceptional divisor consisting of two nonsingular irreducible rational curves with
self-intersections $-1$ and $-n_1-2\leq -5$ and dual graph a cycle. For the
minimal good resolution,
the method in \cite{eo} is no longer applicable; however,
the method in \cite{b} is still applicable and gives
an alternative way of constructing the open
book decomposition $\OB_{(n_1)}$ on $Y_{(n_1)}$.
\end{Rem}

\begin{Rem} According to \cite[Thm 4.3.1]{vhm}, the open book
$\OB_{\overline{n}}$ we constructed on $Y_{\overline{n}}$ is
compatible with a universally tight contact structure. This fact also follows by Lemma~\ref{mil} below.
\end{Rem}

\begin{Lem} \label{mil} The open book $\OB_{\overline{n}}$ on
$Y_{\overline{n}}$ is a Milnor open book.
\end{Lem}

\begin{proof} According to \cite[Theorem 2.1]{np}, there is an analytic structure
$(Z,p)$ on the cone over $Y_{\overline{n}}$ and a corresponding
Milnor open book $\widetilde{\OB}_{\overline{n}}$ on
$Y_{\overline{n}}$ whose binding agrees with the binding of
$\OB_{\overline{n}}$. Now the circular plumbing graph
$\Gamma_{\overline{n}}$ provides a decomposition of
$Y_{\overline{n}}$ into a union $\bigcup V_i$, where $V_i$ is an
$S^1$-bundle over $S^2$ with 2 discs removed. Since any page of
$\OB_{\overline{n}}$ intersects in $V_i$ in exactly one component
for each $i$, by the argument of the proof of Theorem 4.6 of
\cite{cnp}, any horizontal open book whose binding agrees with the
binding of $\widetilde{\OB}_{\overline{n}}$  must be isomorphic to
$\widetilde{\OB}_{\overline{n}}$ . Thus $\OB_{\overline{n}}$ is
indeed a Milnor open book.
\end{proof}

\begin{Prop} \label{uyt} There are precisely two distinct isotopy classes
of Stein fillable universally tight contact structures on
$Y_{\overline{n}}$. One of them is the contact structure induced on
the boundary of the Stein surface depicted in Figure~\ref{hypdiag}
where each Legendrian $2$-handle attains its minimal possible
rotation number. The other one is obtained, similarly,  by achieving
the  maximal possible rotation number for each Legendrian
$2$-handle. Both of these contact structures represent the
isomorphism class of the  canonical contact structure $\xi_{can}$ on
$Y_{\overline{n}}$.
\end{Prop}

\begin{figure}[ht]
  \relabelbox \small {
  \centerline{\epsfbox{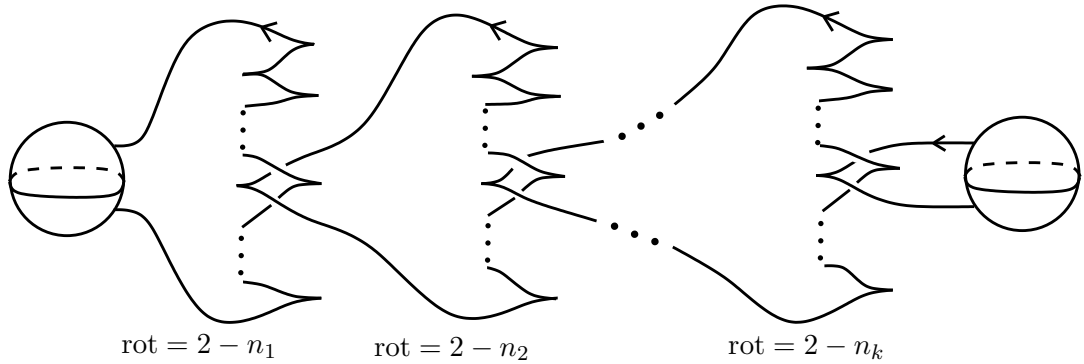}}}

\relabel{r1}{$\rot=2-n_1$}

\relabel{r2}{$\rot=2-n_2$}

\relabel{r3}{$\rot=2-n_k$}

  \endrelabelbox
        \caption{Legendrian handlebody diagram of a Stein filling of the contact $3$-manifold $(Y_{\overline{n}},\xi_{can})$.}
        \label{hypdiag}
\end{figure}

\begin{proof}
On the torus bundle $Y_{\overline{n}}$, there are $(n_1
-1)(n_2-1)\cdots(n_k-1)$ distinct isotopy classes of Stein fillable
contact structures two of which are universally tight according to
Honda's classification \cite{h2}. Each of these contact structures
can be realized as the boundary of some Stein surface as follows:
Consider the Dehn surgery description of the $3$-manifold
$Y_{\overline{n}}$ depicted in Figure~\ref{kirby}$(a)$ and $(b)$,
corresponding to the cases $k >1$ and $k=1$, respectively (cf.
\cite{neum}). By replacing the $0$-framed unknot linking the chain
with a dotted circle---which corresponds to surgering the
corresponding $2$-handle into a $1$-handle (see, for example,  page
168 in \cite{gs})---one obtains a handlebody diagram of a smooth
$4$-manifold whose boundary is diffeomorphic to $Y_{\overline{n}}$.
We opt to represent the unique $1$-handle in the diagram by a pair
of spheres aligned horizontally.

\begin{figure}[ht]
  \relabelbox \small {
  \centerline{\epsfbox{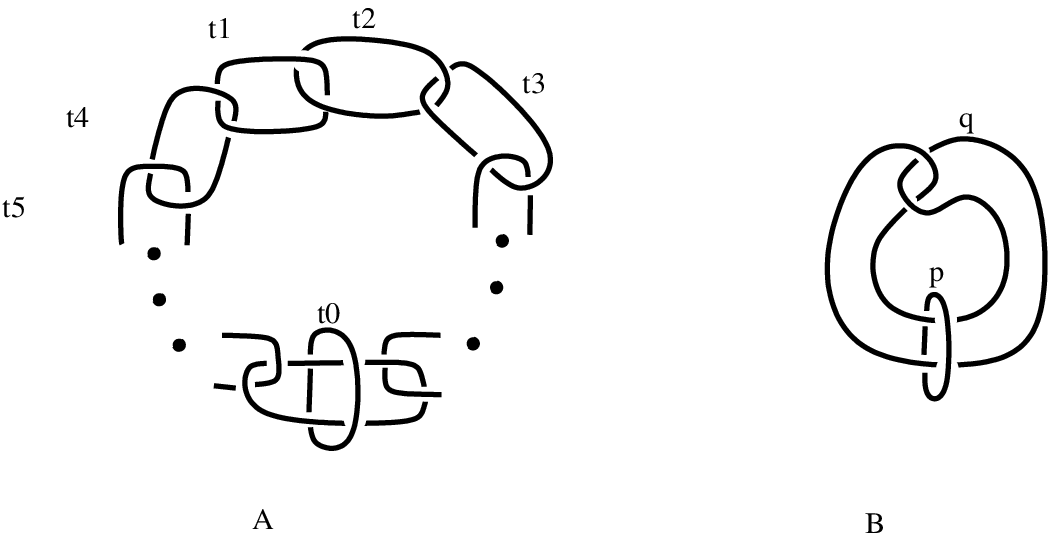}}}

\relabel{t1}{$-n_1$}

\relabel{t2}{$-n_2$}

\relabel{t3}{$-n_3$}

\relabel{t4}{$-n_{k}$}

\relabel{t5}{$-n_{k-1}$}

\relabel{t0}{$0$}

\relabel{p}{$0$}

\relabel{q}{$-n_1+2$}

\relabel{A}{$(a)\;  k > 1$}

\relabel{B}{$(b)\;  k = 1$}

  \endrelabelbox
        \caption{Surgery description of $Y_{\overline{n}}$}
        \label{kirby}
\end{figure}

Next we Legendrian realize all the unknots in this diagram including
the one that goes over the $1$-handle twice with zero linking. For
the case $k>1$, it is easy to see that the unknot with framing $-n_i
\leq -2$ can be realized as a Legendrian unknot whose rotation
number belongs to the set $$\{2-n_i, 4-n_i, \ldots, n_i-4, n_i-2\}$$
\noindent of $n_i-1$ elements, by putting zigzags to the right and
left alternatively. Similarly, for the case $k=1$, one can check
that the unknot with framing $-n_1+2$ which goes over the $1$-handle
twice with zero linking can be realized as a Legendrian unknot whose
rotation number belongs to the set $$\{2-n_1, 4-n_1, \ldots, n_1-4,
n_1-2\}$$ of $n_1-1$ elements, again by putting zigzags to the right
and left alternatively. As a consequence, the finite number of
isotopy classes of Stein fillable contact structures on
$Y_{\overline{n}}$ can be realized as boundaries of certain Stein
surfaces by considering all possible rotation numbers for each
Legendrian unknot in the above described surgery diagram.

We claim that the canonical contact structure must be the one where the
rotation numbers of all the Legendrian unknots are minimized as in
Figure~\ref{hypdiag} (or the one where the rotation numbers of all
the Legendrian unknots are maximized). This claim can be proved
using the argument we used in \cite[Proposition 9.1]{bo}; we provide
details below for the convenience of the reader.

Fix an analytic structure $(X,p)$ on the cone over $Y_{\overline n}$ and
note that the minimal resolution $\pi\colon \tilde X\to X$ provides a holomorphic
filling $(W,J)$ of $(Y_{\overline n},\xi_{can})$. In particular, $W$
is a regular neighborhood of the exceptional
divisor $E=\bigcup E_j$ of $\pi$, the dual graph of which is just the circular
plumbing graph $\Gamma_{\overline n}$. Since the curves $E_j$ are holomorphic, by the adjunction formula,
we have
\begin{equation} \label{holo}
\langle c_1(J),[E_j]\rangle = E_j \cdot E_j - 2\genus(E_j) + 2 = E_j \cdot E_j + 2.
\end{equation}

Now using a result of Bogomolov \cite{bd}, deform the complex structure $J$ so that
$(W,J^\prime)$ becomes a Stein surface, possibly after blowing down some $(-1)$-curves. Since $W$ contains
no topologically embedded spheres of self-intersection number $-1$, $(W,J^\prime)$ itself must be
Stein. Note that \eqref{holo} must continue to hold for $J^\prime$ even
though the curves $E_j$ are no longer holomorphic, since $J$ and $J^\prime$ are
homotopic to one another.

Now let $\{(W^i,J^i)\}$, for $i=1,\ldots,(n_1-1)(n_2-1)\cdots(n_k-1)$, denote
the finite set of Stein fillable contact structures on $Y_{\overline n}$
considered above by taking Legendrian realizations of
the diagram in Figure~\ref{kirby}.
Denote by $U^i_j$ a component of the corresponding Legendrian link
and let $S^i_j$ denote the associated surface in the Stein filling
$(W^i,J^i)$ obtained by pushing a Seifert surface for $U^i_j$ into
the $4$-ball union 1-handle and capping off by the core of the corresponding
$2$-handle (see \cite{go}). Notice that each $W^i$ is diffeomorphic
to $W$ by a diffeomorphism which carries $S^i_j$ to $E_j$ for each
$j$ (see \cite{gs}).

Now, using the well-known identities
$$ S^i_j \cdot S^i_j  = \tb(U^i_j)-1,\qquad \langle c_1(J^i), [S^i_j]\rangle = \rot(U^i_j) $$
(see \cite{go} for the second), observe that $\langle
c_1(J^i),[S^i_j]\rangle = S^i_j \cdot S^i_j + 2$ precisely when
$\rot(U^i_j) = \tb(U^i_j)+1$. Since the latter equality holds
exactly when all the cusps of $U^i_j$ except one are up cusps, it
follows that $\langle c_1(J),[E_j]\rangle = \langle
c_1(J^i),[S^i_j]\rangle$ for each $j$ precisely when all the extra
zigzags are chosen so that the additional cusps are all up cusps,
that is, when all the extra zigzags are chosen on the same fixed
side (which is determined by the orientation of the Legendrian
unknots). The proof is now completed by appealing to Lisca--Mati\`c \cite{lm} and noting that in the finite
list of Stein fillable contact structures on $Y_{\overline n}$ there is only one Stein
fillable contact structure up to isomorphism that comes from a Legendrian
realization as above where all the extra zigzags are on the same
fixed side.

\end{proof}

The $1$-handle in the handlebody diagram depicted in
Figure~\ref{hypdiag} can be replaced by a contact $(+1)$-surgery
along a Legendrian unknot as described in \cite{dg2} to obtain a
contact surgery diagram of $\xi_{can}$ on $Y_{\overline{n}}$  as
shown in Figure~\ref{hypwout}. Notice that the Euler class of
$\xi_{can}$ vanishes by Proposition~\ref{num} although its Stein
filling depicted in Figure~\ref{hypdiag} has non-vanishing first
Chern class.

\begin{figure}[ht]
  \relabelbox \small {
  \centerline{\epsfbox{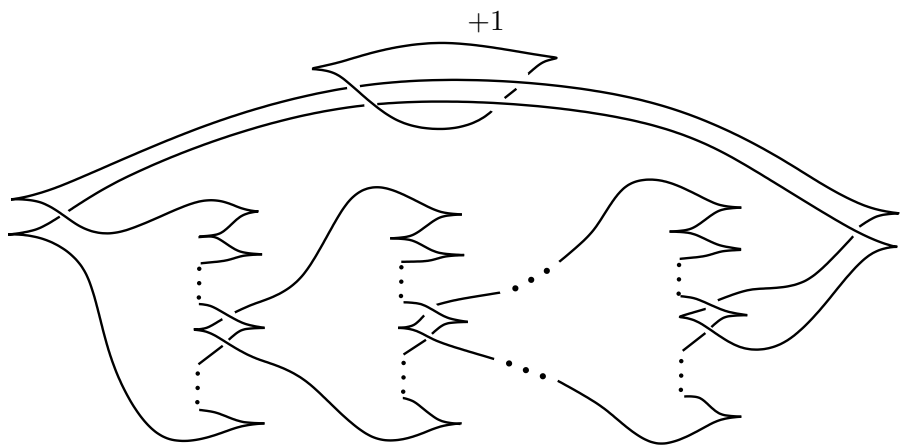}}}




\relabel{a}{$+1$}

  \endrelabelbox
        \caption{A contact surgery diagram for the canonical contact structure $\xi_{can}$ on the link $Y_{\overline{n}}$ of a cusp singularity.}
        \label{hypwout}
\end{figure}

\v \noindent {\bf {Acknowledgement}}: We would like to thank the
referee for his corrections and suggestions that helped us
 improve the article considerably. We would also like to thank
 Patrick Popescu-Pampu for communicating a proof of Lemma~\ref{anti} to the referee.

\end{document}